\documentclass[10pt, a4paper]{iopart}

\expandafter\let\csname equation*\endcsname\relax
\expandafter\let\csname endequation*\endcsname\relax

\usepackage{amsmath, amsthm}
\allowdisplaybreaks

\makeatletter
\@ifpackageloaded{amssymb}{}{\usepackage[bitstream-charter]{mathdesign}}
\makeatother

\usepackage{color}

\usepackage[pdftex, bookmarksopen, pagebackref,
            pdftitle={Tikhonov and Landweber convergence rates: characterization by interpolation spaces},
            pdfauthor={R. Andreev},
            colorlinks, linkcolor=black, urlcolor=black, citecolor=black,
]{hyperref}

\newcommand{\IR}{\mathbb{R}}
\newcommand{\IN}{\mathbb{N}}

\newcommand{\IX}{\mathbb{X}}
\newcommand{\argmax}{\mathop{\mathrm{argmax}}}
\newcommand{\norm}[2]{     \| #1       \|_{ #2 }}

\newcommand{\normiii}[2]{\vert\kern-0.25ex\vert\kern-0.25ex\vert #1 \vert\kern-0.25ex \vert\kern-0.25ex\vert_{ #2 }}
\newcommand{\Normiii}[2]{\left\vert\kern-0.25ex\left\vert\kern-0.25ex\left\vert #1 \right\vert\kern-0.25ex\right\vert\kern-0.25ex\right\vert_{ #2 }}
\newcommand{\scalar}[2]{     \langle #1       \rangle_{ #2 }}

\newtheorem{lemma}{Lemma}[section]
\newtheorem{proposition}[lemma]{Proposition}

\theoremstyle{remark}
\theoremstyle{definition}

\newtheorem{example}[lemma]{Example}

\definecolor{someblue}{rgb}{0,0,.8}

\providecommand{\REV}[1]{{#1}}
\newcommand\sectiom{\vfil\penalty-9999\vfilneg\section}
\usepackage{suffix}
\WithSuffix\newcommand\sectiom*{\vfil\penalty-9999\vfilneg\section*}

\begin{document}

\title%
[\hfill Tikhonov and Landweber rates by interpolation spaces \hfill]%
{Tikhonov and Landweber convergence rates: characterization by interpolation spaces}

\author{R.~Andreev}

\address{%
	Univ Paris Diderot, Sorbonne Paris Cit\'e, LJLL (UMR 7598 CNRS), F-75205 Paris, France
}
\ead{roman.andreev@upmc.fr}
\vspace{10pt}
\begin{indented}
\item[]\today
\end{indented}

\begin{abstract}
	Algebraic convergences rates 
	of (iterated) Tikhonov regularization
	for linear inverse problems in Hilbert spaces
	are 
	characterized by 
	the membership of 
	the exact solution 
	to
	intermediate spaces
	produced by the K-method of real interpolation.
	Similar results are obtained for the Landweber iteration.
\end{abstract}


\vspace{2pc}


\section{Introduction}
\label{s:0}

It has long been known
that 
the H\"older type 
source conditions 
originally used
to obtain algebraic convergence rates
for Tikhonov regularization
are only sufficient,
not necessary.
Consequently,
other concepts 
have been introduced 
to completely characterize those rates
such as 
the spectral decay condition of \cite[Theorem 2.1]{Neubauer1997},
distance functions \cite[p.3]{FlemmingHofmannMathe2011}, 
and
variational source conditions 
\cite{Flemming2012};
pointers to the origins of those concepts can be found therein.
We propose here a more concise condition
through 
the notion of interpolation spaces
and 
establish links 
to the concepts just listed.
Specifically,
we argue that
the intermediate spaces
$(E_0, E_1)_{\theta, q}$
produced by the $K$-method of real interpolation
\cite{BerghLofstrom1976, BennettSharpley1988}
with fine index $q = \infty$
naturally capture 
the essential behavior
of (iterated) Tikhonov regularization,
that is:
convergence rates, converse results, and saturation,
in the noise-free and the noisy case.
This is not unexpected, given 
the resemblance of the $K$-functional \eqref{e:K}
and 
the Tikhonov functional \eqref{e:J},
but we systematically quantify this connection
by careful estimates
with particular attention to the limiting cases
$\theta = 0, 1$.
In a similar vein, 
the relationship between near-minimizers for ``$L$-functionals''
and
Tikhonov regularization
was highlighted
in \cite[Chapter 6]{KislyakovKruglyak2013}.
We prove analogous convergence and converse results for 
the Landweber iteration,
and comment on the applicability of the discrepancy principle
as a stopping rule.

This note consists of two main parts.
In the first part (Section \ref{s:i})
we develop the required preliminaries
of the $K$-method of real interpolation,
introduce the different intermediate subspaces,
and describe them and their interrelations using 
spectral theory in Hilbert spaces.
In passing, we relate
to the concepts of distance functions 
and 
variational source conditions.
In the second part (Section \ref{s:ip})
we elaborate on how
the convergence rates of
Tikhonov regularization
and
Landweber iteration
are characterized 
in terms of the intermediate subspaces.

We write $A \lesssim B$
to mean that there is a constant $C \geq 0$ 
independent of the parameters specified by quantifiers
such that $A \leq C B$.
If, in addition, $B \lesssim A$, we write $A \sim B$.

\section{Preliminaries}
\label{s:i}

\subsection{Interpolation spaces}

Let $X$ be a Banach space (here and henceforth: over the reals).
Let $X_1 \subset X$ be another Banach space,
continuously 
embedded in $X$.
We write $\norm{\cdot}{0} := \norm{\cdot}{}$ and $\norm{\cdot}{1}$
for the norms of $X$ and $X_1$, respectively.
The $K$-functional
is defined as
\begin{align}
\label{e:K}
	K_t(x) :=
	\inf_{x_1 \in X_1}
	(
		\norm{x - x_1}{0}^2
		+
		t^2
		\norm{x_1}{1}^2
	)^{1/2}
	,
	\quad x \in X,
	\quad t > 0.
\end{align}
For real $0 < \theta < 1$ and $1 \leq q \leq \infty$
the $K$-method of real interpolation
defines intermediate subspaces 
$X_1 \subset (X, X_1)_{\theta, q} \subset X$
based on the integrability of $t \mapsto K_t(x)$.
Here we are only interested in the case $q = \infty$ 
with one of the spaces embedded into the other,
and therefore
refer to standard sources
such as 
\cite{BerghLofstrom1976, BennettSharpley1988}
for general definitions.
For any $0 \leq \theta \leq 1$ and any $x \in X$ set
\begin{align}
\label{e:||nu}
	\norm{x}{\theta:1}
	:=
	\sup_{t > 0}
	t^{-\theta}
	K_t(x)
	,
\end{align}
possibly infinite.
We point out
that 
we include the limiting cases $\theta = 0$ and $\theta = 1$
in this definition.
The reason for the unusual notation 
will become apparent in Section \ref{s:i:H}
where 
instead of $X_1$
we will 
consider a family of subspaces $X_\gamma \subset X$
parametrized by $\gamma$.
Define the spaces
\begin{align}
	X_{\theta:1} 
	:=
	( X, X_1 )_{\theta, \infty}
	:=
	\{
		x \in X :
		\norm{x}{\theta:1} < \infty
	\}
\end{align}
with the norm $\norm{\cdot}{\theta:1}$.
For $0 < \theta < 1$,
these are Banach spaces.
Moreover,
the following embeddings 
$
	X_1 \subset X_{\theta:1} \subset X
$
are continuous.
The space $X_{\theta:1}$ need not coincide with $X_\theta$
for $\theta = 0,1$,
see the remarks on the Gagliardo completion in \cite[Chapter 5, \S1]{BennettSharpley1988},
but it will be the case in 
the more specific setting of Section \ref{s:i:H}.

\subsection{The constant \texorpdfstring{$N_\theta$}{Nnu}}

For $0 \leq \theta \leq 1$, 
the constant $1 \leq N_\theta \leq \sqrt{2}$ given by
\begin{align}
\label{e:N=sup}
	N_{\theta}^{-2}
	:=
	\theta^\theta (1-\theta)^{1-\theta}
	=
	\sup_{\lambda \in [0, 1]}
	\lambda^\theta (1 - \lambda)^{1 - \theta}
	=
	\sup_{s > 0}
	\frac{s^{2(1 - \theta)}}{s^2 + 1}
\end{align}
will play a recurrent role.
By convention, $N_0 := N_1 := 1$,
making $\theta \mapsto N_\theta$ continuous on $[0, 1]$.
For example,
if
$a, b > 0$
then
Young's inequality with exponents $p = 1/(1-\theta)$ and $q = 1 / \theta$
gives
\begin{align}
\label{e:Nab}
	N_{\theta}^2
	a^{1-\theta} b^{\theta} 
	=
	(a / (1-\theta))^{1-\theta} (b / \theta)^{\theta}
	\leq
	a + b
	.
\end{align}
As another example,
for any real $\lambda \geq 0$ and $t > 0$,
any real $k \geq 1/2$, and any $0 \leq \theta \leq 1$,
\begin{align}
\label{e:Ncomm}
	N_\theta^{2 (1 - 2 k)} 
	\frac{t^{2 (1 - \theta)}}{t^2 + \lambda^{2 k}}
	\leq
	\left(
		\frac{\alpha^{1 - \theta}}{\alpha + \lambda}
	\right)^{2k}
	\quad\text{for}\quad
	\alpha 
	=
	t^{1/k} \left( \frac{1 - \theta}{\theta} \right)^{1 - 1/(2k)}
	.
\end{align}
Indeed, after algebraic simplification, 
the inequality in
\eqref{e:Ncomm}
is equivalent to
$
	(\alpha + \lambda)^p
	\leq
	(1-\theta)^{1-p} \alpha^p + \theta^{1-p} \lambda^p
$,
itself a consequence of H\"older's inequality
with $p = 2k \geq 1$ as one of the exponents.


\subsection{Distance function}
\label{s:i:dr}

Distance functions were introduced
as a means
to characterize the regularization error of linear regularization operators
in \cite{FlemmingHofmannMathe2011},
previously also in \cite[Theorem 2.12]{BakushinskyKokurin2004}.
We briefly comment on the relation to the $K$-functional.
Let $0 < \theta < 1$.
Fix $x \in X$.
Define
the distance function
\begin{align}
\label{e:dr} 
	d(r) 
	:=
	\inf
	\{ \norm{ x - x_1 }{0} : x_1 \in X_1, \; \norm{x_1}{1} \leq r \}
	,
	\quad
	r \geq 0
	.
\end{align}
This function is nonnegative,
nonincreasing,
bounded by $d(0) = \norm{x}{0}$,
and convex.
%
%
%
%
To simplify the notation,
we shall write $\norm{x_1}{1} \leq r$, or similar,
without mentioning that $x_1 \in X_1$.
It is clear the $d$ has compact support
if and only if $x \in X_1$.

Inspection of the definitions 
of the $K$-functional \eqref{e:K} and the distance function \eqref{e:dr}
reveals
that
\begin{align}
\label{e:Kd*}
	K_t(x) = \inf_{r > 0} (d(r)^2 + t^2 r^2)^{1/2}
	\sim
	\inf_{r > 0} (d(r) + t r)
	=
	-d^*(-t)
	,
\end{align}
where $d^*$ is the Legendre--Fenchel conjugate of $d$.
Thus the distance function \eqref{e:dr} 
characterizes the subspace $X_{\theta:1} \subset X$.
The first characterization,
boundedness
of $|t^{-\theta} d^*(-t)|$,
is obvious from \eqref{e:||nu} and \eqref{e:Kd*}.
The second characterization
is the behavior of $d$ at infinity,
more precisely the identity
\begin{align}
\label{e:D=E}
	E 
	=
	N_{\theta}^{-1}
	\norm{x}{\theta:1}
	=
	D 
	,
\end{align}
where
\begin{align}
\label{e:ED}
 	E
	:=
	\sup_{r > 0}
	\inf_{ \norm{x_1}{1} \leq r }
	\norm{ x - x_1 }{0}^{1-\theta} 
	\norm{x_1}{1}^\theta
	\quad\text{and}\quad
 	D
	:=
	\sup_{r > 0}
	d(r)^{1 - \theta} 
	r^\theta
	.
\end{align}
From \eqref{e:D=E} one infers
the more qualitative observation
that
$x \in X_{\theta:1}$
if and only if
the distance function \eqref{e:dr} exhibits
the asymptotic decay rate
$d(r) = \mathcal{O}(r^{-\theta / (1-\theta)})$ for $r \to \infty$.

\begin{proof}[{Proof of \eqref{e:D=E}}]
	The proof 
	is in the three steps:

	a) $E \leq N_{\theta}^{-1} \norm{x}{\theta:1}$.
	%
	%
	Estimating the $\inf$ of $E$ as in \eqref{e:Nab}
	for $a := \norm{ x - x_1 }{0}^2$ and $b := t^2 \norm{x_1}{1}^2$,
	\begin{align*}
		\inf_{\norm{x_1}{1} \leq r} (\cdots)
	\leq
		N_{\theta}^{-1}
		t^{-\theta}
		\inf_{\norm{x_1}{1} \leq r}
		(
			\norm{x - x_1}{0}^2
			+
			t^2 \norm{x_1}{1}^2
		)^{1/2}
	\leq
		N_{\theta}^{-1}
		\sup_{t > 0}
		t^{-\theta}
		K_t(x)
		.
	\end{align*}
	Here,
	the condition $\norm{x_1}{1} \leq r$ is redundant
	if $t$ is large enough,
	so the infimum becomes $K_t(x)$,
	explaining the second inequality.
	Taking the supremum over $r > 0$ proves the claim.
	
	b) $\norm{x}{\theta:1} \leq N_{\theta} D$.
	The definition \eqref{e:ED} of $D$ 
	implies
	$d(r) \leq (r^\theta / D)^{1/(\theta - 1)}$.
	Using this in \eqref{e:Kd*},
	then
	computing the infimum
	yields
	$K_t(x) \leq t^\theta N_{\theta} D$.
	Now
	multiply by $t^{-\theta}$
	and take $\sup_{t > 0}$.

	c) $D \leq E$.
	If $D$ is finite then $d(\infty) = 0$.
	Given that $d$ is convex,
	$d$ is strictly decreasing
	(unless where it vanishes).
	Thus,
	for each $r > 0$,
	the infimum in $d(r)$
	is achieved at $\norm{x_1}{1} = r$.
	Concerning $E$,
	we may suppose 
	that 
	the $\sup \inf$ is assumed at $\norm{x_1}{1} = r$,
	adjusting $r$ if necessary.
	Hence, both $D$ and $E$ equal
	$
		\sup_{r > 0} 
		\inf_{ \norm{x_1}{1} = r }
		\norm{ x - x_1 }{0}^{1-\theta} 
		\norm{x_1}{1}^\theta
	$.
	This establishes \eqref{e:D=E}.
\end{proof}

\subsection{Interpolation inequality of operators}

Let $Y$ be a Banach space.
Let $S : X \to Y$ be a bounded linear operator with norm $C_0 \geq 0$.
Assume that 
$S|_{X_1} : X_1 \to Y$ 
is also a bounded linear operator, with norm $C_1 \geq 0$.
Then, for any $0 < \theta < 1$, 
\begin{align}
\label{e:|Snu|}
	\norm{S x}{Y}
	\leq
	N_{\theta}
	C_0^{1-\theta} C_1^\theta
	\norm{x}{\theta:1}
	\quad
	\forall x \in X_{\theta:1}
	.
\end{align}
%
%
Indeed,
let $x \in X_{\theta:1}$.
Given any $x_1 \in X_1$, 
write $x = (x - x_1) + x_1$,
apply the triangle inequality with boundedness of $S$,
and estimate by Cauchy--Schwarz:
$ 
	\norm{S x}{Y}
	\leq
	C_0 \norm{x - x_1}{0}
	+
	C_1 \norm{x_1}{1}
	\leq
	t^\theta (C_0^2 + t^{-2} C_1^2)^{1/2}
	\times
	t^{-\theta}
	(\norm{x - x_1}{0}^2 + t^2 \norm{x_1}{1}^2)^{1/2}
	.
$ 
Inserting 
an
infimum over $x_1 \in X_1$ 
then a supremum over $t > 0$
in the second factor,
in view of \eqref{e:||nu}
we obtain
$ 
	\norm{S x}{Y}
	\leq
	t^\theta (C_0^2 + t^{-2} C_1^2)^{1/2}
	\times
	\norm{x}{\theta:1}
	.
$ 
Minimization over $t > 0$ gives \eqref{e:|Snu|}.

%

\subsection{A lemma for measures}
\label{s:}

We will 
call a nonnegative finite measure $\mu$ on Borel subsets of $[0, \infty)$
a ``Borel measure on $[0, \infty)$''.
For any such $\mu$ and any real $\nu \geq 0$
we define
$\normiii{\mu}{\nu}$
by
\begin{align}
	\normiii{\mu}{\nu}^2
	:=
	\sup_{t > 0} t^{-2\nu} \mu([0, t))
	.
\end{align}

The following Lemma 
records
two useful properties of $\normiii{\cdot}{\nu}$.

\begin{lemma}
\label{l:mu}
	Let $\mu$ be a Borel measure on $[0, \infty)$.
	Let $\nu > 0$.
	Then
	\begin{align}
	\label{e:Tail2}
		\mu([0, \Lambda))
		+
		\Lambda^{2 \gamma}
		\int_{[\Lambda, \infty)}
		\lambda^{-2 \gamma}
		d\mu(\lambda)
		\leq
		\frac{\gamma}{\gamma - \nu}
		\Lambda^{2 \nu}
		\normiii{\mu}{\nu}^2
		\quad
		\forall \Lambda > 0
		\quad
		\forall \gamma > 0
		,
	\end{align}
	and
	\begin{align}
	\label{e:Tail1}
		\int_{[0, \Lambda)} \lambda^{-2 r} d\mu(\lambda)
	\leq
		\frac{r + \nu}{\nu}
		\Lambda^{2 \nu}
		\normiii{\mu}{r + \nu}^2
	\quad
		\forall \Lambda > 0
		\quad
		\forall r \geq 0
		,
	\end{align}
	whenever the right-hand-side is finite.
\end{lemma}

\begin{proof}
	Define the left-continuous function
	$I(\Lambda) := \mu([0, \Lambda))$
	for $\Lambda \geq 0$.
	Fix $\Lambda > 0$.
	Writing the integral as a Riemann--Stieltjes integral, 
	and 
	integrating by parts we have
	\begin{align}
		\int_{[\Lambda, \infty)}
		\lambda^{-2 \gamma}
		d\mu(\lambda)
		=
		\int_{[\Lambda, \infty)}
		\lambda^{-2 \gamma}
		d I(\lambda)
		=
		-\Lambda^{-2 \gamma} I(\Lambda) 
		+
		2 \gamma
		\int_{[\Lambda, \infty)}
		\lambda^{-2 \gamma - 1}
		I(\lambda)
		d \lambda
		.
	\end{align}
	Estimating 
	$I(\lambda) \leq \lambda^{2 \nu} \normiii{\mu}{\nu}^2$ 
	under the integral,
	evaluating,
	and rearranging
	leads to \eqref{e:Tail2}.
	Similarly,
	\begin{align}
		\int_{[0, \Lambda)}
		\lambda^{-2 r}
		d \mu(\lambda)
	\leq
		\int_{[0, \Lambda]} 
		\lambda^{-2 r}
		d I(\lambda)
	=
		\left. 
			\lambda^{-2 r} I(\lambda) 
		\right|_{\lambda = 0}^{\lambda \searrow \Lambda}
		+
		2 r
		\int_{[0, \Lambda]}
		\lambda^{-2 r - 1}
		I(\lambda)
		d \lambda
		.
	\end{align}
	Estimating $I(\lambda) \leq \lambda^{2 (r + \nu)} \normiii{\mu}{r + \nu}^2$
	and evaluating the integral yields \eqref{e:Tail1}.
\end{proof}

{%
\begin{proof}[Alternative proof] 
	Fix $\Lambda > 0$.
	Ad \eqref{e:Tail2}:
	Consider
	$A_s := [\Lambda, \infty) \cap \{ \lambda \geq 0 : \lambda^{-2 \gamma} > s \}$.
	If $s \geq \Lambda^{-2 \gamma}$ then $A_s$ is empty,
	otherwise
	$A_s = [\Lambda, s^{-1/(2\gamma)})$.
	Therefore
	\begin{align}
		\int_{[\Lambda, \infty)} \lambda^{-2 \gamma} d\mu(\lambda)
		=
		\int_0^{\Lambda^{-2 \gamma}} \mu(A_s) ds
		=
		-\Lambda^{-2\gamma} \mu([0, \Lambda))
		+
		\int_0^{\Lambda^{-2 \gamma}} \mu([0, s^{-1/(2\gamma)})) ds
	\end{align}
	Estimating $\mu([0, t)) \leq t^{2 \nu} \normiii{\mu}{\nu}^2$ under the integral
	and evaluating
	yields \eqref{e:Tail2}.
	Ad \eqref{e:Tail1}:
	The statement is trivial for $r = 0$, so suppose $r > 0$.
	Consider 
	$B_s := [0, \Lambda) \cap \{ \lambda \geq 0 : \lambda^{-2 r} > s \}$.
	If $s \leq \Lambda^{-2r}$ 
	then $B_s = [0, \Lambda)$ 
	and 
	$\mu(B_s) \leq \Lambda^{2 (\nu + r)} \normiii{\mu}{\nu + r}^2$.
	Otherwise
	$B_s = [0, s^{-1/(2r)})$ ,
	so that
	$\mu(B_s) \leq s^{-(\nu + r)/r} \normiii{\mu}{\nu + r}^2$.
	Using this in 
	$
		\int_{[0, \Lambda)} \lambda^{-2 r} d\mu(\lambda)
	=
		\int_0^\infty \mu(B_s) ds
	$
	yields \eqref{e:Tail1}.
	%
\end{proof}
}
	
\subsection{Spectral theory in Hilbert spaces}
\label{s:i:H}

Suppose that $X$ and $Y$ are real Hilbert spaces.
Let $T : X \to Y$ be a nonzero bounded linear operator.
Replacing $X$ by $X / \ker T$ if necessary
(with the usual quotient norm),
we assume that
\begin{align}
\label{e:inj}
	\text{$T$ is injective}
	.
\end{align}
Let $E$ denote the projection valued spectral measure of $T^* T$.
Then $E$ is compactly supported in $[0, \infty)$ since $T$ is bounded,
and injectivity of $T$ is equivalent to $E_{\{0\}} = 0$,
since 
$\mathop{\mathrm{range}}(E_{\{\lambda\}}) = \ker( T^* T - \lambda I)$.
For $x \in X$ we define 
the Borel measure $\mu_x$ on $[0, \infty)$
by
$\mu_x(A) := \norm{E_A x}{}^2$
and
the associated quantity
\begin{align}
\label{e:normiii}
	\normiii{x}{\nu} 
	:= 
	\normiii{\mu_x}{\nu}
	=
	\sup_{t > 0}
	t^{-\nu}
	\norm{ E_{[0, t)} x }{}
	,
	\quad
	\nu \geq 0
	.
\end{align}
The subset 
\begin{align}
	\IX_\nu :=
	\{
		x \in X:
		\normiii{x}{\nu} < \infty
	\}
\end{align}
of $X$
is indeed a Banach space 
equipped with the norm
$\normiii{\cdot}{\nu}$.
A description of this subspace 
as an interpolation space
is subsequently given 
in 
Proposition \ref{p:ga-nu}.
The definition of $\IX_\nu$ is inspired by 
the work \cite{Neubauer1997}.

For all real $\gamma \geq 0$,
we define the Banach space $X_\gamma \subset X$ 
as
\begin{align}
\label{e:Xga}
	X_\gamma := \mathop{\mathrm{range}} ( (T^* T)^\gamma )
	\quad\text{with the norm}\quad
	\norm{x}{\gamma}
	:=
	\norm{ (T^* T)^{-\gamma} x }{},
	\quad
	x \in X_\gamma
	.
\end{align}
%
%
%
%
%
In terms of the spectral measure $E$, we can write
(note $E_{\{0\}} = 0$)
\begin{align}
\label{e:0ga}
	\norm{x}{0}^2
	=
	\int_{[0, \infty)} d \mu_{x}(\lambda)
	\quad\text{and}\quad
	\norm{x}{\gamma}^2
	=
	\int_{[0, \infty)} \lambda^{-2\gamma} d \mu_{x}(\lambda)
	\quad
	\forall
	x \in X_\gamma
	.
\end{align}

For any real $0 \leq \nu \leq \gamma$,
we define
the interpolation space
\begin{align}
	X_{\nu : \gamma}
	:=
	( X, X_\gamma )_{\nu/\gamma, \infty}
	\quad\text{with norm}\quad
	\norm{\cdot}{\nu : \gamma}
	,
\end{align}
and will denote the corresponding $K$-functional by $K_t^\gamma$.
One can check
(most easily in the case that $T$ is compact)
that
for any $x \in X$ and $t > 0$,
\begin{align}
\label{e:KtI}
	| K_t^\gamma(x) |^2
=
	\int_{[0, \infty)}
	\inf_{\epsilon > 0}
	\{
		(1 - \epsilon)^2
		+
		\epsilon^2 
		t^2 
		\lambda^{-2 \gamma}
	\}
	d\mu_{x}(\lambda)
=
	\int_{[0, \infty)}
	\frac{t^2}{t^2 + \lambda^{2 \gamma}}
	d\mu_{x}(\lambda)
	.
\end{align}
For the limiting cases $\nu = 0$ and $\nu = \gamma$
we recover from \eqref{e:0ga} and \eqref{e:KtI}
that
\begin{align}
\label{e:X01}
	X_{0 : \gamma} = X_0
	\quad\text{and}\quad
	X_{\gamma : \gamma} = X_\gamma
	\quad
	\text{with equality of norms}
	.
\end{align}
%
%
%
%


The different spaces are related by the following Proposition.

\begin{proposition}
\label{p:ga-nu}
	Let $0 \leq \nu < \gamma$.
	%
	%
	Let $x \in X$.
	Then
	\begin{align}
	\label{e:ga-nu}
		\sqrt{ 1 - \nu / \gamma }
		\norm{x}{\nu : \gamma}
	\stackrel{a)}{\leq}
		\normiii{x}{\nu}
	\stackrel{b)}{\leq}
		N_{\nu / \gamma}
		\norm{x}{\nu : \gamma}
	\stackrel{c)}{\leq}
		\norm{x}{\nu}
	\stackrel{d)}{\leq}
		\norm{T^* T}{}^{\gamma - \nu}
		\norm{x}{\gamma}
	.
	\end{align}
	Hence, the following embeddings are continuous:
	\begin{align}
		X_{\gamma}
		\subset
		X_\nu
		\subset
		X_{\nu : \gamma}
		=
		\IX_\nu
		,
	\end{align}
	where
	$X_{\nu : \gamma} = \IX_\nu$ 
	with equivalence of norms 
	possibly not uniform in $\nu < \gamma$.
\end{proposition}
\begin{proof}
	The inequality
	(\ref{e:ga-nu}d) follows
	from
	the representation \eqref{e:0ga},
	restricting the domain of integration
	to $[0, \norm{T^* T}{}]$.
	The inequality (\ref{e:ga-nu}c)
	is obtained by 
	identifying $s := t \lambda^{-\gamma}$ and $\theta := \nu / \gamma$
	in \eqref{e:N=sup},
	and using it in \eqref{e:KtI},
	viz.~%
	\begin{align}
		\norm{x}{\nu : \gamma}^2 
		= 
		\sup_{t > 0} t^{-2 \nu/\gamma} |K_t^\gamma(x)|^2
		\leq
		N_{\nu / \gamma}^{-2}
		\int_{[0, \infty)} \lambda^{-2 \nu} d \mu_x(\lambda)
		.
	\end{align}
	For the inequality (\ref{e:ga-nu}b)
	we use the identity
	\begin{align}
		s^{-2\nu}
		=
		N_{\nu/\gamma}^2
		t^{-2\nu/\gamma}
		\frac{ t^2 }{ t^2 + s^{2 \gamma} }
		\quad\text{for}\quad
		t = 
		s^\gamma
		\sqrt{ \frac{\gamma - \nu}{\nu} }
		,
	\end{align}
	combined with $\lambda \leq s$
	in the first step of 
	\begin{align*}
		s^{-2\nu}
		\norm{
			E_{[0, s)} x
		}{}^2
		\leq
		N_{\nu/\gamma}^2
		t^{-2\nu/\gamma}
		\int_{[0, s)}
		\frac{ t^2 }{ t^2 + \lambda^{2 \gamma} }
		d\mu_x(\lambda)
		\stackrel{\eqref{e:KtI}}{\leq}
		N_{\nu/\gamma}^2
		t^{-2\nu/\gamma}
		| K_t^\gamma(x) |^2
		.
	\end{align*}
	Taking
	$\sup_{t > 0}$ on the right,
	then
	$\sup_{s > 0}$ on the left
	gives (\ref{e:ga-nu}b).
	Finally,
	from \eqref{e:KtI}, 
	followed by \eqref{e:Tail2},
	we have 
	\begin{align*}
		\norm{x}{\nu:\gamma}^2
	& \leq
		\sup_{t > 0} t^{-2 \nu/\gamma}
		\inf_{s > 0}
		\left\{
			\norm{
				E_{[0, s)} x
			}{}^2
			+
			s^2 \int_{[s, \infty)} \lambda^{-2\gamma} d \mu_x(\lambda)
		\right\}
	\leq
		\frac{\gamma}{\gamma - \nu}
		\normiii{x}{\nu}^2
		,
	\end{align*}
	with the choice $s := t^{1/\gamma}$ for the last inequality,
	and this shows (\ref{e:ga-nu}a).
	The last claim is a combination of (\ref{e:ga-nu}a) and (\ref{e:ga-nu}b).
\end{proof}


%

%
The choice $\gamma := 2 \nu$ in \eqref{e:ga-nu}
leads to the chain of inequalities
\begin{align}
\label{e:infg}
	\tfrac{1}{\sqrt{2}} \normiii{\cdot}{\nu}
	\leq
	\inf_{\gamma > \nu} \norm{\cdot}{\nu : \gamma}
	\leq
	\norm{\cdot}{\nu : 2 \nu}
	\leq
	\sqrt{2} \normiii{\cdot}{\nu}
	,
\end{align}
and therefore,
$X_{\nu : 2 \nu} = \IX_\nu$ 
with equivalence of norms uniformly in $\nu \geq 0$.
Equality between
$\inf_{\gamma > \nu} \norm{\cdot}{\nu : \gamma}$
and
$\norm{\cdot}{\nu : 2 \nu}$
does not hold in general.
However, 
if $\mu_x$ is a Dirac measure supported at some $\lambda_0 > 0$, 
and $\nu > 0$,
then the infimum of
\begin{align}
	\norm{x}{\nu : \gamma}^2
	=
	\sup_{t > 0}
	t^{-2 \nu/\gamma}
	\frac{t^2}{t^2 + \lambda_0^{2 \gamma}}
	\stackrel{\eqref{e:N=sup}}{=}
	N_{\nu / \gamma}^{-2}
	\lambda_0^{-2 \nu}
	,
\end{align}
over $\gamma > \nu$
is indeed achieved at $\gamma = 2 \nu$.

The constants in (\ref{e:ga-nu}a) and (\ref{e:ga-nu}b) are sharp.
For example,
for
$\mu_x = \delta_{\lambda_0}$ being the Dirac measure at $\lambda_0 = 1$,
\begin{align}
	\normiii{x}{\nu}
	=
	1
	\quad\text{and}\quad
	N_{\nu / \gamma}
	\norm{x}{\nu : \gamma}
	=
	1
	.
\end{align}
On the other hand,
for
$d \mu_x(\lambda) = 2 \nu \lambda^{2 \nu - 1} dx$
(that this measure is not compactly supported is not essential)
we find
$\normiii{x}{\nu} = 1$,
while
\begin{align}
	\norm{x}{\nu : \gamma}^2
	= 
	\sup_{t > 0} t^{-2\nu} |K_t^\gamma(x)|^2 = 
	\frac{\pi \nu / \gamma}{\sin(\pi \nu / \gamma)}
	=
	(1 + o(1))
	\frac{\gamma}{\gamma - \nu}
	\quad\text{as}\quad
	\nu \nearrow \gamma
	,
\end{align}
so that
the norm equivalence in \eqref{e:ga-nu}
does 
deteriorate
as $\nu \nearrow \gamma$.
The relation
to
the finite qualification of the (iterated) Tikhonov regularization
is discussed in
Section \ref{s:ip:T}.

We provide next another illustration
of the fact that $\IX_\nu$ is in general strictly larger than $X_\nu$,
here for $\nu := 1$.
We will revisit this example in Section \ref{s:ip:L}, Example \ref{x:IX:LW}.

\begin{example}
\label{x:IX}
	Consider the diagonal operator 
	$T := \mathop{\mathrm{diag}} ((n^{-1/2})_{n \geq 1})$ on
	the sequence space $X := \ell_2(\IN)$.
	Then $T^* T$ has eigenvalues $\lambda_n = n^{-1}$, $n \geq 1$.
	Let $x^\dag := (n^{-3/2})_{n \geq 1}$.
	Then $\mu_{x^\dag} = \sum_{n \geq 1} n^{-3} \delta_{\lambda_n}$,
	so that
	\begin{align}
		\normiii{x^\dag}{1}^2 
		=
		\sup_{N \geq 1}
		\lambda_N^{-2} \mu_{x^\dag}([0, \lambda_N]) 
		= 
		\sup_{N \geq 1}
		N^2
		\sum_{n \geq N} n^{-3}
		< \infty
		,
	\end{align}
	yet,
	$\norm{x^\dag}{1}^2 = \sum_{n \geq 1} n^{-1}$ is not finite.
	Therefore $x^\dag \in \IX_1 \setminus X_1$.
	The $\sup$ is assumed at $N = 1$, 
	so that
	$\normiii{x^\dag}{1}^2 = \zeta(3) \approx (1.096)^2$.
\end{example}

\providecommand{\IX}{\mathbb{X}}
In \cite[Proposition 11]{AndreevElbauDeHoopLiuScherzer2015}
the variational inequality
\begin{align}
\label{e:S}
	\exists \beta \geq 0:
	\quad
	| \scalar{ x, \omega }{} |
	\leq
	\beta
	\norm{ (T^* T)^{\gamma} \omega }{}^{\nu / \gamma}
	\norm{ \omega }{}^{1 - \nu / \gamma}
	\quad
	\forall \omega \in X
	,
\end{align}
was shown 
for $0 < \nu < \gamma$
to hold 
if and only if
$x \in \IX_\nu$.
We prove a more precise statement,
in particular including
the limiting cases $\nu = 0$ and $\nu = \gamma$.
The first part of the proof
(the inequality ``$\leq$'')
simplifies and sharpens
the corresponding part of 
\cite[Proof of Proposition 11]{AndreevElbauDeHoopLiuScherzer2015}.
The second part (the inequality ``$\geq$'')
draws from \cite{Flemming2012}.

\begin{proposition}
\label{p:svi}
	Let $x \in X$ and $0 \leq \nu \leq \gamma$.
	Then
	\begin{align}
	\label{e:p:svi}
		\sup_{\norm{\omega}{} = 1}
		\norm{(T^* T)^{\gamma} \omega}{}^{-\nu/\gamma}
		|\scalar{x, \omega}{}|
		=
		N_{\nu / \gamma}
		\norm{x}{\nu:\gamma}
	.
	\end{align}
	%
\end{proposition}

\begin{proof}
	For $\nu = 0$ the statement is trivial
	due to
	$\norm{x}{0:\gamma} = \norm{x}{}$.
	For $\nu = \gamma$
	the statement follows from \cite[Lemma 8.21]{ScherzerBook2009}
	and $\norm{x}{\gamma:\gamma} = \norm{x}{\gamma}$.
	In both cases, recall $N_0 = N_1 = 1$.
	For the remainder of the proof 
	we assume $0 < \nu < \gamma$.

	Let $\omega \in X$ 
	and
	consider the linear mapping $S : x \mapsto \scalar{ x, \omega }{}$.
	Then
	$| S x | \leq \norm{\omega}{} \norm{x}{}$,
	and
	$| S x | \leq \norm{(T^* T)^{\gamma} \omega}{} \norm{x}{\gamma}$
	for all $x \in X_{\gamma}$.
	By the operator interpolation inequality \eqref{e:|Snu|}
	we have
	$
		| \scalar{ x, \omega }{} |
	\leq
		N_{\nu/\gamma}
		\norm{x}{\nu : \gamma}
		\norm{(T^* T)^{\gamma} \omega}{}^{\nu/\gamma}
		\norm{\omega}{}^{1 - \nu/\gamma}
	$.
	This implies ``$\leq$'' in \eqref{e:p:svi}.

	To verify ``$\geq$'' in \eqref{e:p:svi},
	it suffices to establish the case $\gamma = 1$,
	then apply it 
	with $(T^* T)^\gamma$ replacing $T^* T$ 
	(also in the definitions of the norms in Section \ref{s:i:H}).
	Thus we assume that \eqref{e:S} holds with $0 < \nu < \gamma = 1$.
	To verify $N_{\nu} \norm{x}{\nu:1} \leq \beta$
	we check
	$D \leq N_{\nu}^{-2} \beta$
	for the quantity $D$ from \eqref{e:D=E}.
	In other words
	we
	show the bound
	$d(r)^{1 - \nu} r^\nu \leq N_{\nu}^{-2} \beta$
	for
	the distance function $d(\cdot)$ from \eqref{e:dr}.
	To that end we 
	combine \cite[Proposition 2.10]{Flemming2012},
	\cite[Theorem 4.1]{Flemming2012}
	and \cite[Theorem 4.5]{Flemming2012}:
	The inequality
	\begin{align}
		| \scalar{ x, \omega }{} |
		\leq
		\beta
		\norm{ (T^* T) \omega }{}^{\frac{\kappa}{2 - \kappa}}
		\quad
		\forall \omega \in X
		\quad\text{with}\quad
		\norm{\omega}{} = 1,
	\end{align}
	with 
	\begin{align}
		\kappa := 
		\tfrac{2 \nu}{1 + \nu}
		\quad\text{and}\quad
		\beta :=
		\tfrac{2 - \kappa}{2}
		\left(
			\tfrac{\tilde\beta}{1 - \kappa} 
		\right)^{\frac{1 - \kappa}{2 - \kappa}}
		a^{\frac{1}{2 - \kappa}}
		\quad
		\text{%
			for some $0 < \tilde\beta \leq 1$ and $a > 0$,
		}
	\end{align}
	implies
	$
		d^2(r) \leq 
		\tilde\beta \, (-\varphi)^*(-2 r)
	$
	for all $r > 0$,
	where $(-\varphi)^*$ is the Legendre--Fenchel conjugate
	of $t \mapsto -\varphi(t) := -a t^\kappa$ defined for $t > 0$.
	Straightforward but tedious algebra
	yields
	the desired estimate 
	$d(r)^{1 - \nu} r^\nu \leq N_{\nu}^{-2} \beta$.
\end{proof}

%




%

\section{Application to linear inverse problems in Hilbert spaces}
\label{s:ip}

\subsection{Linear inverse problem}
\label{s:ip:0}

Let $X$ and $Y$ be Hilbert spaces.
Let $T : X \to Y$ be a bounded linear operator,
with possibly nonclosed range.
We write $\norm{\cdot}{}$ for the norm of $X$ and for that of $Y$.
As in Section \ref{s:i:H},
we assume that $T$ is injective.
Fix $y^\dag \in T X$ and 
let $x^\dag$ denote the solution to
\begin{align}
\label{e:Tx=y}
	T x^\dag = y^\dag
	.
\end{align}
%
Let
$ 
	x_0 \in X,
$ 
called a prior,
be given.
The task is to find an approximation of $x^\dag$,
given $y^\dag$ (noise-free case)
or $y^\delta \approx y^\dag$ (noisy case)
with
\begin{align}
\label{e:ydel}
	\norm{ y^\dag - y^\delta }{}
	\leq
	\delta
	.
\end{align}

%
We use the notation
from Section \ref{s:i:H},
including
the spectral measure $E$, 
the Borel measure $\mu_x$, 
the spaces $X_\gamma$, $X_{\nu:\gamma}$, etc.
%

\subsection{Spectral cut-off regularization}
\label{s:ip:cut}

The spectral cut-off regularization of $x^\dag$ is defined
as $x_0 + E_{[\alpha, \infty)} (x^\dag - x_0)$
for a parameter $\alpha > 0$.
From the definition \eqref{e:normiii} of $\normiii{\cdot}{\nu}$
it is immediate that
the error of this regularization is 
\begin{align}
\label{e:coerr}
	\norm{(I - E_{[\alpha, \infty)}) (x^\dag - x_0)}{}
	=
	\norm{E_{[0, \alpha)} (x^\dag - x_0)}{} 
	\leq 
	\alpha^\nu \normiii{x^\dag - x_0}{\nu}
	\quad
	\forall \alpha \geq 0
	\quad
	\forall \nu \geq 0
	.
\end{align}
Since there are no restrictions on the possible convergence rate $\nu \geq 0$
(referred to as infinite qualification),
and no further constants are involved,
we may view the performance of this regularization as a reference.

\subsection{Tikhonov regularization}
\label{s:ip:T}

For $\alpha > 0$, 
the regularized solution $x_\alpha \in X$ in the noise-free case
is defined 
as the unique minimizer of the Tikhonov functional
\begin{align}
\label{e:J}
	J_\alpha(x; x_0)
	:=
	\norm{ y^\dag - T x }{}^2
	+
	\alpha
	\norm{ x - x_0 }{}^2,
	\quad
	x \in X
	.
\end{align}
Replacing $y^\dag$ by $y^\delta$
defines
the regularized solution $x_\alpha^\delta \in X$ 
in the noisy case.
They are equivalently characterized by
the first order optimality conditions
\begin{align}
\label{e:xx}
	x_\alpha = (T^* T + \alpha I)^{-1} (T^* y^\dag + \alpha x_0)
	\quad\text{and}\quad
	x_\alpha^\delta = (T^* T + \alpha I)^{-1} (T^* y^\delta + \alpha x_0)
	.
\end{align}

Writing $e^\delta := x_\alpha - x_\alpha^\delta$ for the moment,
we have
$ 
	\norm{ T e^\delta }{}^2 + \alpha \norm{ e^\delta }{}^2
=
	\scalar{ (T^* T + \alpha I) e^\delta, e^\delta }{}
=
	\scalar{ T^* (y^\dag - y^\delta), e^\delta }{}
\leq
	\norm{ y^\dag - y^\delta }{} \norm{ T e^\delta }{}
\leq
	\tfrac14 \delta^2 + \norm{ T e^\delta }{}^2
	,
$ 
and cancellation of $\norm{ T e^\delta }{}^2$ on both ends gives
the error splitting 
\begin{align}
\label{e:xad}
	\norm{ x^\dag - x_\alpha^\delta }{}
	\leq
	\norm{ x^\dag - x_\alpha }{}
	+
	\norm{ e^\delta }{}
	\leq
	\norm{ x^\dag - x_\alpha }{}
	+
	\tfrac12
	\tfrac{\delta}{\sqrt{\alpha}}
	.
\end{align}
%

The parameter $\alpha > 0$ is determined by
a parameter choice strategy
\begin{align}
\label{e:pcs}
	\bar{\alpha} : 
	(\delta, y^\delta, \ldots) \mapsto \bar{\alpha}(\delta, y^\delta, \ldots)
	.
\end{align}
The one that minimizes 
$\alpha \mapsto \norm{x^\dag - x_\alpha^\delta}{}$
whenever $y^\delta$ and $y^\dag$ are fixed
may be considered the optimal strategy.
We shall suppose that
the parameter choice strategy 
satisfies
\begin{align}
\label{e:qo}
	\delta^{-2\nu}
	\sup_{{{\eqref{e:ydel}}}}
	\norm{x^\dag - x_{\bar{\alpha}(\delta, y^\delta, \ldots)}^\delta}{}^{2\nu + 1}
	\sim
	\sup_{\alpha > 0}
	\alpha^{-\nu}
	\norm{x^\dag - x_{\alpha}}{}
	\quad
	\forall
	\delta > 0
	,
\end{align}
where the hidden constants 
do not depend on $\delta$, the exact solution $x^\dag$, or the prior $x_0$,
but may depend on $\nu \geq 0$.
Here, $\sup_{{{\eqref{e:ydel}}}}$
means the supremum over all $y^\delta \in Y$
which satisfy \eqref{e:ydel}.
As an example,
the a priori parameter choice strategy
$\bar{\alpha}$
defined by
\begin{align}
\label{e:pcs1}
	\sqrt{\bar{\alpha}(\delta, y^\dag)}
	\norm{ x^\dag - x_{\bar{\alpha}(\delta, y^\dag)} }{}
	\stackrel{{!}}{=}
	\tfrac12
	\delta
	\quad
	\forall \delta > 0
\end{align}
satisfies \eqref{e:qo}.
Specifically, 
the error splitting
\eqref{e:xad} quickly yields
$\text{LHS} \leq 2 \, \text{RHS}$ in \eqref{e:qo}
and
an inspection of
\cite[Proof of Theorem 2.6]{Neubauer1997}
yields
$\frac{13}{2} (5 / \sqrt{2})^{2 \nu + 1} \text{LHS} \geq \text{RHS}$.
That proof
also shows that the optimal strategy, see above,
satisfies \eqref{e:qo}.
Of course, of practical interest are 
parameter choice strategies
that do not access 
the exact data $y^\dag$ or the exact solution $x^\dag$;
\REV{%
in that regard 
the notion 
of quasioptimality by Raus \& H\"amarik \cite{RausHamarik2007}
is useful,
see comments following Proposition \ref{p:tik} below.
}
%
%
In any case, 
\eqref{e:qo} formalizes
the equivalence
\begin{align}
	\text{``}
		\norm{ x^\dag - x_{\bar{\alpha}(\delta, y^\delta, \ldots)}^\delta }{}
		\lesssim
		\delta^{2 \nu / (2 \nu + 1)}
		\quad
		\forall \delta > 0
	\text{''}
	\quad\Leftrightarrow\quad
	\text{``}
		\norm{x^\dag - x_\alpha}{}
		\lesssim
		\alpha^\nu
		\quad
		\forall \alpha > 0
	\text{''}
\end{align}
of the error estimates 
in the noisy and in the noise-free 
cases,
and
in the following we assume \eqref{e:qo},
and focus on its \text{RHS}.
%

%

From the abstract theory of interpolation,
convergence rates of 
the Tikhonov regularization error $\norm{x^\dag - x_\alpha}{}$
for $x^\dag \in X_{\nu:1}$ when $0 < \nu < 1$
quickly follow.
Indeed,
for $\alpha > 0$
consider the linear mapping
$S_\alpha : X \to X$,
$(x^\dag - x_0) \mapsto (x^\dag - x_\alpha)$.
From \eqref{e:xx},
\begin{align}
\label{e:Sa}
	S_\alpha = \alpha (T^* T + \alpha I)^{-1}
	,
\end{align}
so that
$\norm{S_\alpha }{} \leq 1$.
%
%
Under the classical source condition
\begin{align}
\label{e:cla}
	x^\dag \in 
	x_0 + X_1
	\quad\text{where}\quad
	X_1 = 
	\mathop{\mathrm{range}} (T^* T)
	,
\end{align}
it is known
(and shown below in \eqref{e:p:tik} for $\nu = 1$)
that
\begin{align}
\label{e:a01}
	\norm{ S_\alpha (x^\dag - x_0) }{} 
	\leq 
	\alpha \norm{x^\dag - x_0}{1}
	.
\end{align}
The operator interpolation inequality 
\eqref{e:|Snu|} implies
\begin{align}
\label{e:anu}
	\norm{ x^\dag - x_\alpha }{} 
	=
	\norm{ S_\alpha (x^\dag - x_0) }{}
	\leq
	\alpha^\nu
	N_\nu
	\norm{ x^\dag - x_0 }{\nu:1}
	\quad
	\forall x^\dag \in x_0 + X_{\nu:1}
	\quad
	\forall \alpha > 0
	.
\end{align}
The following Proposition
shows that $x_0 + X_{\nu:1} \subset X$ is precisely
the (affine) subspace that allows those convergence rates,
which is the main observation of this note.

\begin{proposition}
\label{p:tik}
	Let $0 \leq \nu \leq 1$. Then
	\begin{align}
	\label{e:p:tik}
		N_\nu^{-1}
		\norm{x^\dag - x_0}{\nu:1}
		\leq
		\sup_{\alpha > 0}
		\alpha^{-\nu}
		\norm{ x^\dag - x_\alpha }{}
		\leq
		\norm{x^\dag - x_0}{\nu:1}
		.
	\end{align}
	Equalities hold for $\nu = 0$ and $\nu = 1$.
	If $\nu = 1$ then $\sup_{\alpha > 0}$ can be replaced by $\lim_{\alpha \searrow 0}$.
\end{proposition}

The result is a special case of Proposition \ref{p:tikk} below,
and the proof is therefore omitted.
Several remarks are in order.

	
	Combining (\ref{e:ga-nu}a)--(\ref{e:ga-nu}b) and \eqref{e:p:tik},
	for $0 < \nu < 1$,
	we have
	that $\norm{x^\dag - x_\alpha}{} \leq C \alpha^{\nu}$ for all $\alpha > 0$
	if and only if
	$x^\dag \in x_0 + \IX_\nu$.
	In essence,
	this was already shown in \cite[Theorem 2.1]{Neubauer1997}.
	We emphasize, however,
	that (\ref{e:ga-nu}a) and \eqref{e:p:tik} yield
	the more precise upper bound
	\begin{align}
	\label{e:bc}
		\norm{x^\dag - x_\alpha}{} 
		\leq 
		\alpha^\nu \tfrac{1}{\sqrt{1 - \nu}} 
		\normiii{x^\dag - x_0}{\nu}
		\quad
		\forall \alpha > 0 
		.
	\end{align}
	In particular, 
	although the rate of convergence of Tikhonov regularization
	is at least $\nu$ whenever $x^\dag \in x_0 + \IX_\nu$,
	the constant in \eqref{e:bc} may deteriorate as $\nu \nearrow 1$
	compared to 
	the error \eqref{e:coerr}
	of the spectral cut-off regularization.
	This may be interpreted as a quantitative description of 
	the finite qualification of Tikhonov regularization,
	that is its inability to provide convergence larger than $\nu = 1$.
	Of course,
	by \eqref{e:p:tik},
	the rate of $\nu = 1$ does hold 
	if (and only if) $x^\dag \in x_0 + X_1$,
	but
	this condition is more restrictive 
	than $x^\dag \in x_0 + \IX_1$,
	cf.~%
	Proposition \ref{p:ga-nu}
	and
	Example \ref{x:IX}.
	
	%


	\REV{%
	From \eqref{e:p:tik}
	we see that
	Tikhonov regularization with 
	a parameter choice rule satisfying \eqref{e:qo}
	is order optimal
	on 
	$M_{\nu,\rho} := \{ x \in X : \norm{x}{\nu:1} \leq \rho \}$
	if $T$ has nonclosed range.
	Indeed, \eqref{e:qo} and \eqref{e:p:tik}
	imply
	the error estimate
	$
		\norm{x^\dagger - x_{\bar{\alpha}}^\delta }{}
		\lesssim
		\delta^{2\nu / (2\nu + 1)}
		\rho^{1 / (2\nu + 1)}
	$
	whenever
	$x^\dagger \in M_{\nu,\rho}$;
	on the other hand,
	\cite[Prop 3.15]{EnglHankeNeubauer1996}
	and comments there
	show that
	this estimate is optimal
	because
	$M_{\nu,\rho}$ contains
	the classical source set
	$\{ x \in X : \norm{x}{\nu} \leq \rho' \}$
	by (\ref{e:ga-nu}c).
	Now we can invoke
	\cite[Thm 2.3]{RausHamarik2007}
	to assert that
	any parameter choice rule 
	that is strongly quasioptimal
	in the sense of \cite[Def 2.2]{RausHamarik2007}
	will again give rise 
	to
	an order optimal method on $M_{\nu,\rho}$.
	}


	The final statement of Proposition \ref{p:tik}
	implies the saturation result:
	If $\norm{x^\dagger - x_\alpha}{} = o(\alpha)$
	as $\alpha \searrow 0$
	then $x^\dagger = x_0$.
	An analogous statement holds 
	in the noisy case 
	for any
	parameter choice strategy
	satisfying \eqref{e:qo}.
	

	An additional consequence 
	of \eqref{e:p:tik}
	is that,
	for $0 \leq \nu \leq \gamma$,
	\begin{align}
	\label{e:c:tikr}
		N_{\nu/\gamma}^{-1}
		\norm{ x^\dag - x_0 }{\nu : \gamma}
		\leq
		\sup_{\alpha > 0}
		\alpha^{-\nu}
		\norm{ x^\dag - x_{\alpha:\gamma} }{}
		\leq
		\norm{ x^\dag - x_0 }{\nu : \gamma}
	\end{align}
	where $x_{\alpha:\gamma}$ is given by
	\begin{align}
	\label{e:c:tikr:x}
		x_{\alpha:\gamma} 
		:= 
		\mathop{\mathrm{argmin}}_{x \in X} 
		\{
			\norm{ (T^* T)^{\gamma/2} (x^\dag - x) }{}^2
			+
			\alpha
			\norm{ x^\dag - x_0 }{}^2
		\}
		.
	\end{align}
	This is simply \eqref{e:p:tik}
	for the operator
	$(T^* T)^{\gamma/2}$ instead of $T$.
	Therefore,
	letting $x_0 = 0$,
	the mapping
	$x^\dag \mapsto \sup_{\alpha > 0} \alpha^{-\nu} \norm{x^\dag - x_{\alpha:\gamma}}{}$
	defines a norm on $X_{\nu:\gamma}$
	which 
	is equivalent to $\norm{\cdot}{\nu:\gamma}$
	uniformly in $0 \leq \nu \leq \gamma$.
	%


	The variational inequality \eqref{e:S}
	was used as a starting point in 
	\cite[Lemma 2(ii) \& Lemma 7]{AndreevElbauDeHoopLiuScherzer2015}
	for an elementary proof of
	convergence rates for Tikhonov regularization 
	without 
	invoking spectral theory.
	This is certainly of interest,
	but leads to the natural question 
	whether \eqref{e:S} itself
	can be established in particular cases 
	without spectral theory; 
	the characterization \eqref{e:p:svi}
	in terms of interpolation spaces
	is a step in that direction.
	That elementary proof, however,
	seems to be limited to rates $\nu \leq 1/2$ 
	in \eqref{e:anu},
	as it uses the variational inequality \eqref{e:S} 
	with $\gamma = 1/2$.

	
	As an example
	consider
	the following integral operator
	on
	square integrable functions
	$X := Y := L_2$
	of the unit interval $(-1, 1)$,
	\begin{align}
		\textstyle
		T : X \to Y,
		\quad
		(T x)(t)
		:=
		\int_{-1}^{t} x(s) ds
		-
		\frac{1}{2}
		\int_{-1}^1 
			\int_{-1}^\tau x(s) ds
		d\tau
		.
	\end{align}
	One can check that
	$u := T^* T x$
	solves
	the boundary value problem
	$-u'' = x$, $u|_{\pm1} = 0$.
	Hence, 
	$X_1$ from \eqref{e:Xga}
	is 
	the Sobolev space $H^2 \cap H_0^1$.
	%
	%
	Let now $y^\dagger(t) := |t| - 1$,
	so that \eqref{e:Tx=y}
	holds
	with $x^\dagger(s) := \mathop{\mathrm{sign}}(s)$.
	This means, 
	solving \eqref{e:Tx=y} with noisy data $y^\delta$
	amounts to taking the derivative of $y^\dagger$
	perturbed by $L_2$ noise.
	Set $\nu := 1/4$.
	The interpolation space
	$X_{\nu:1} = (X, X_1)_{1/4, \infty}$
	is the Besov space $\mathring{B}_{2, \infty}^{1/2}$
	\cite[Prop 1.25 and 1.30]{Guidetti1991}.
	In view of \cite[Def 1.12 of $\mathring{B}$]{Guidetti1991},
	to check that $x^\dagger \in X_{\nu:1}$,
	it suffices to check that
	$L_2$ step functions are in 
	$B_{2,\infty}^{1/2}(\IR)$,
	but
	this is immediate from
	the intrinsic definition 
	\cite[Def 1.1]{Guidetti1991}.
	On the other hand,
	\cite[Ch 1, Thm 11.7]{LionsMagenes1972}
	implies that
	$t \mapsto |x(t)|^2 / (t^2 - 1)$ is integrable
	for any $x \in X_{1/4}$
	(denoted by $H_{00}^{1/2}$ in \cite{LionsMagenes1972}),
	which is clearly not the case for $x^\dagger$.
	In summary,
	$x^\dagger \in X_{(1/4):1} \setminus X_{1/4}$,
	and 
	Proposition \ref{p:tik}
	implies 
	the  
	convergence rate $\nu = 1/4$,
	familiar from \cite[Example 5]{FlemmingHofmannMathe2011}.

\REV{%
	On a more general note,
	let $M$ be
	a connected smooth finite-dimensional Riemannian manifold
	with bounded geometry
	(complete, injectivity radius bounded away from zero, bounded covariant derivatives of the curvature tensor). 
	For instance, 
	the flat space $\IR^d$
	or
	any compact manifold without boundary
	has bounded geometry.
	On $M$,
	Sobolev $H^s$ and Besov $B_{2,\infty}^s$
	spaces can be defined 
	intrinsically through local means \cite[\S1.11.4]{Triebel1992},
	and by \cite[\S1.11.3]{Triebel1992}
	one has
	$(H^s, H^{s+k})_{\theta, \infty} = B_{2,\infty}^{s + \theta k}$.
	Thus, 
	for $T^* T$ with domain $X := H^s$ and range $X_1 = H^{s + 2m}$,
	$m > 0$,
	we have
	$X_{\nu:1} = B_{2,\infty}^{s + \nu 2 m}$,
	so that the convergence rate $\nu$ in \eqref{e:anu}
	is characterized in terms of Besov smoothness of $x^\dagger$.
	This covers a wide class of operators $T$,
	for instance 
	bijective elliptic pseudo-differential operators
	$H^s \to H^{s+m}$
	of order $-m$
	such as
	$(I - \Delta)^{-m/2}$.
}
	

\subsection{Stationary iterated Tikhonov regularization}
\label{s:ip:IT}

Fixing $\alpha > 0$,
for each integer $k \geq 1$,
define $x_{\alpha, k}$ as the minimizer of
the Tikhonov functional
$x \mapsto J_\alpha(x; x_{\alpha,k-1})$,
where $J_\alpha$ is given by \eqref{e:J}
and
$x_{\alpha,0} := x_0$.
Thus, each iterate is used as a prior 
for the next one
with the same (stationary) choice of 
the regularization parameter $\alpha > 0$.
Then 
\begin{align}
\label{e:errTk}
	x^\dagger - x_{\alpha,k} = S_\alpha^k (x^\dagger - x_0)
	\quad
	\forall k \geq 1,
\end{align}
with $S_\alpha$ from \eqref{e:Sa}.
%
%
%
%
%
Since $S_\alpha$ commutes with powers of $(T^* T)$,
we obtain from \eqref{e:a01}
the ``shift estimate''
$
	\norm{ S_\alpha (x^\dag - x_\alpha) }{k - \ell} 
	\leq 
	\alpha \norm{x^\dag - x_0}{k - \ell + 1}
$,
and repeating this argument for each iteration $\ell = 1, 2, \ldots, k$,
of the $k$-fold Tikhonov regularization
then
$ 
	\norm{ S_\alpha (x^\dag - x_{\alpha, k}) }{} 
	\leq
	\alpha^k
	\norm{x^\dag - x_0}{k}
	.
$ 
The operator interpolation inequality 
readily yields
the convergence rate
$\alpha^{\nu}$ 
for $x^\dag \in x_0 + X_{\nu:k}$ and $0 \leq \nu \leq k$.
The following extension of 
Proposition \ref{p:tik}
is more precise.

\begin{proposition}
\label{p:tikk}
	Let $k \geq 1$ be an integer, let $0 \leq \nu \leq k$ be real.
	Then
	\begin{align}
	\label{e:p:tikk}
		N_{\nu / k}^{1 - 2 k}
		\norm{ x^\dag - x_0 }{\nu : k}
		\leq
		\sup_{\alpha > 0}
		\alpha^{-\nu}
		\norm{ x^\dag - x_{\alpha,k} }{}
		\leq
		\norm{ x^\dag - x_0 }{\nu : k}
		.
	\end{align}
	%
	%
	If $\nu = k$ then $\sup_{\alpha > 0}$ can be replaced by $\lim_{\alpha \searrow 0}$.
\end{proposition}
\begin{proof}
	From \eqref{e:errTk} and \eqref{e:Sa},
	\begin{align}
	\label{e:p:tikk:a}
		\alpha^{-2 \nu}
		\norm{ x^\dag - x_{\alpha,k} }{}^2
	 	=
	 	\alpha^{-2 \nu}
	 	\int_{[0, \infty)}
	 	\left( \frac{\alpha}{\alpha + \lambda} \right)^{2 k}
	 	d \mu_{x^\dag - x_0}(\lambda)
		.
	\end{align}
	The first inequality in \eqref{e:p:tikk}
	is therefore a consequence of \eqref{e:Ncomm}
	with $\theta = \nu / k$,
	and the representation \eqref{e:KtI} of the interpolation norm.
	The second inequality in \eqref{e:p:tikk}
	follows by identifying $\alpha = t^{1/k}$,
	estimating
	$1 / (\alpha + \lambda)^{2 k} \leq 1 / ( t^2 + \lambda^{2 k} )$,
	and
	invoking the representation \eqref{e:KtI}.
	Finally,
	if $\nu = k$,
	then
	$\sup_{\alpha > 0}$\eqref{e:p:tikk:a} = 
	$\lim_{\alpha \searrow 0}$\eqref{e:p:tikk:a}
	by
	the Lebesgue monotone convergence theorem.
\end{proof}


%

Remarks analogous to those following Proposition \ref{p:tik} apply.
For instance,
the $k$-fold Tikhonov regularization
saturates:
If $\norm{x^\dag - x_{\alpha, k}}{} = o(\alpha^k)$ as $\alpha \searrow 0$
then $x^\dag = x_0$.

%

\subsection{Landweber iteration}
\label{s:ip:L}

Fix $\sigma > 0$ with
\begin{align}
\label{e:tau}
	0 < \sigma \norm{T^* T}{} \leq 1
	.
\end{align}
The Landweber iterates $x_k$ are defined by
\begin{align}
\label{e:LW}
	x_k :=
	x_{k-1} + \sigma T^* (y^\dag - T x_{k-1}),
	\quad
	k \in \IN,
	\quad\text{with a given}\quad
	x_0 \in X,
\end{align}
and $x_k^\delta$ by the same iteration with $y^\dag$ replaced by $y^\delta$
in the noisy case.
%
%
%
%
%
%
One motivation for this iteration
is that
$x^\dag$ is a fixed point.
%
By induction
one finds
\begin{align}
	x^\dag - x_k
	=
	(I - \sigma T^* T)
	(x^\dag - x_{k-1})
	=
	(I - \sigma T^* T)^k 
	(x^\dag - x_0)
	\quad
	\forall
	k \in \IN
	,
\end{align}
and similarly
the residual representation
\begin{align}
\label{e:dade}
	y^\dag - T x_k
	=
	(I - T T^*)^k (y^\dag - T x_0)
	\quad\text{and}\quad
	y^\delta - T x_k^\delta
	=
	(I - T T^*)^k (y^\delta - T x_0)
	.
\end{align}
The condition \eqref{e:tau} 
therefore guarantees
nondivergence of the iterates.

For the noisy case, 
an error splitting analogous to \eqref{e:xad}
is true \cite[Lemma 6.2]{EnglHankeNeubauer1996}:
\begin{align}
\label{e:LW:es}
	\norm{ x^\dag - x_k^\delta }{}
	\leq
	\norm{ x^\dag - x_k }{}
	+
	\delta
	\sqrt{k}
	\quad
	\forall k \geq 0
	,
\end{align}
thus one often ``morally'' identifies 
$k$ with $1/\alpha$.
We call a mapping 
\begin{align}
\label{e:sr}
	\bar{k}
	:
	(\delta, y^\delta, \ldots)
	\mapsto
	\bar{k}
	(\delta, y^\delta, \ldots)
\end{align}
a stopping rule,
and
in analogy to \eqref{e:pcs}--\eqref{e:qo}
we shall suppose that
it satisfies
\begin{align}
\label{e:qok}
	\delta^{-2\nu}
	\sup_{{{\eqref{e:ydel}}}}
	\norm{x^\dag - x_{\bar{k}(\delta, y^\delta, \ldots)}^\delta}{}^{2\nu + 1}
	\sim
	\sup_{k \geq 0}
	(1 + k/\nu)^{\nu}
	\norm{x^\dag - x_k}{}
	\quad
	\forall
	\delta > 0
	,
\end{align}
where the hidden constants 
do not depend on $\delta$, $x^\dag$, or $x_0$,
but may depend on $\nu \geq 0$.
The factor $(1 + k/\nu)$ 
is motivated by \eqref{e:LWE}--\eqref{e:p:LW}.
The stopping rule $\bar{k}$ may be based 
on the knowledge of some of the iterates $x_k$,
for example 
it may be the smallest $k \geq 0$ for which
the 
discrepancy principle
\begin{align}
\label{e:MDPk}
	\norm{
		y^\delta - T x_k^\delta
	}{}
	\leq
	\delta \tau
\end{align}
is satisfied with some fixed threshold $\tau > 1$,
which in particular is not allowed to depend on $x^\dag$.
We will comment on this stopping rule at the end of this section.
%
%

%
We shall show that 
convergence rates of 
the Landweber iteration \eqref{e:LW}
in the noise-free case
characterize the spaces $\IX_\nu$.
Before formalizing this, we need a lemma.

\begin{lemma}
\label{l:c2}
	The constant 
	\begin{align}
	\label{e:c2}
		c_2
		:=
		\sup_{a > 0}
		\sup_{s > 0}
		\sqrt{ I(s, a) },
		\quad
		I(s, a)
		:=
		\left(
			\frac{ s + a }{ a }
		\right)^a
		\int_0^1 s (1 - t)^{s-1} t^{a} dt,
	\end{align}
	is finite
	with
	$c_2 \approx 1.135$.
\end{lemma}
\begin{proof}
	First, $I(s, a)$ is well-defined
	for all real $s > 0$ and $a > 0$.
	The integral, 
	known as the beta function \cite[\S6.2]{AbramowitzStegun1972},
	evaluates to 
	$\Gamma(s+1) \Gamma(a+1) / \Gamma(s+a+1)$.
	One has $\lim_{s \searrow 0} I(s, a) = 1$ for all $a > 0$.
	Moreover,
	the function
	$s \mapsto I(s, a)$
	is increasing for $0 \leq a \leq 1$
	and
	decreasing for $1 \leq a$.
	So we need to consider 
	the pointwise limit $I(s, a)$ as $s \to \infty$
	for $0 \leq a \leq 1$.
	Stirling's formula \cite[\S6.1.38]{AbramowitzStegun1972} 
	for $\Gamma(\cdot)$ yields
	\begin{align}
		\label{e:soo}
		\bar{I}(a) := \lim_{s \to \infty} I(s, a) = a^{-a} \Gamma(a + 1)
		\quad
		\forall a > 0
		,
	\end{align}
	which clearly is bounded in $0 \leq a \leq 1$.
	We find
	numerically
	that
	$\bar{I}$ is maximized at $a^* \approx 0.3164$
	with $\bar{I}(a^*) \approx 1.288$.
	The constant $c_2$ is the square root of the latter.
\end{proof}

We can now state 
the
announced 
rate characterization for 
the Landweber iteration.

\begin{proposition}
\label{p:LW}
	Let $\{ x_k \}_{k \geq 0}$ be the Landweber iterates 
	generated by \eqref{e:LW}
	with 
	step size $\sigma > 0$ as in \eqref{e:tau}.
	Let $\nu > 0$ and $r \geq 0$.
	Set
	\begin{align}
	\label{e:LWE}
		\Delta_{\nu}^r(x^\dag - x_0)
		:=
		\sup_{k \geq 0}
		\varepsilon_{k}^{-(r + \nu)}
		\norm{ (T^* T)^r (x^\dag - x_k) }{}
		\quad\text{where}\quad
		\varepsilon_{k}
		:=
		\frac{1}{\sigma}
		\frac{r + \nu}{k + r + \nu}
		.
	\end{align}

	Then, 
	%
	\begin{align}
	\label{e:p:LW}
		c_1
		\sqrt{ \tfrac{\nu}{r + \nu} }
		\normiii{x^\dag - x_0}{\nu}
	\leq
		\Delta_{\nu}^r(x^\dag - x_0)
	\leq
		c_2
		\normiii{x^\dag - x_0}{\nu}
		,
	\end{align}
	where $c_2$ is from \eqref{e:c2} and
	\begin{align}
	\label{e:p:LW:c1}
		c_1
		:=
		\inf_{k \geq 0}
		\left\{
			(1 - \sigma \varepsilon_{k})^{k}
			\left[
				\varepsilon_{k+1} / \varepsilon_{k}
			\right]^{r + \nu}
		\right\}
		\geq
		( \sigma \varepsilon_1 / e)^{r + \nu}
		.
	\end{align}
\end{proposition}


\begin{proof}
	Fix $x^\dag \in X$.
	We abbreviate
	$\mu := \mu_{x^\dag - x_0}$ 
	and
	$\Delta := \Delta_{\nu}^r(x^\dag - x_0)$.
	The quantity $\varepsilon$ has been defined
	such as
	to satisfy
	\begin{align} 
		\varepsilon_{k}
		=
		\argmax_{\lambda \geq 0}
		f(\lambda)
		\quad\text{with}\quad
		f(\lambda) := \lambda^{2(r + \nu)} (1 - \sigma \lambda)^{2k}
		\quad
		\forall k \geq 0
		.
	\end{align}
	To enhance readability we will assume 
	for the remainder of the proof
	that $x_0 = 0$ and $\sigma = 1$.

	For the first inequality of \eqref{e:p:LW},
	define $\mu^{2r}(A) := \int_A \lambda^{2 r} d \mu(\lambda)$
	on Borel subsets $A \subset [0, \infty)$.
	Given $0 < \Lambda \leq \norm{T^* T}{}$,
	determine the integer $k \geq 0$ 
	by
	$\varepsilon_{k+1} < \Lambda \leq \varepsilon_{k}$.
	Then
	\begin{align}
		\mu^{2 r}([0, \Lambda])
		\leq
		\mu^{2 r}([0, \varepsilon_{k}])
	& \leq
		(1 - \varepsilon_{k})^{-2 k}
		\int_{[0, \varepsilon_{k}]}
		\lambda^{2 r}
		(1 - \lambda)^{2 k}
		d\mu(\lambda)
	\\
	& \leq
		(1 - \varepsilon_{k})^{-2 k}
		\varepsilon_{k}^{2 (r + \nu)}
		\Delta^2
		\quad
		\text{(by definition of $\Delta$)}
	\\
	& \leq
		(1 - \varepsilon_{k})^{-2 k}
		\left[
			\varepsilon_{k} / \varepsilon_{k+1}
		\right]^{2 (r + \nu)}
		\Lambda^{2(r + \nu)}
		\Delta^2
	\\
	&
	\leq
		c_1^{-2} 
		\Lambda^{2(r + \nu)}
		\Delta^2
		.
	\end{align}
	%
	Since $\Lambda > 0$ was arbitrary,
	we obtain the bound
	\begin{align}
	\label{e:p:LW:mu2r}
		\normiii{\mu^{2 r}}{r + \nu}
	\leq 
		c_1^{-1}
		\Delta
		.
	\end{align}
	The first inequality of \eqref{e:p:LW}
	now
	follows from:
	\begin{align}
		\normiii{\mu}{\nu}^2
	& =
		\sup_{t > 0}
		t^{-2 \nu}
		\int_{[0, t)} 
		\lambda^{-2 r}
		d\mu^{2 r}(\lambda)
	\stackrel{\eqref{e:Tail1}}{\leq}
		\tfrac{r + \nu}{\nu}
		\normiii{\mu^{2 r}}{r + \nu}^2
	\stackrel{\eqref{e:p:LW:mu2r}}{\leq}
		\tfrac{r + \nu}{\nu}
		c_1^{-2}
		\Delta^2
		.
	\end{align}

	We now prove the second inequality in \eqref{e:p:LW}.
	In view of \eqref{e:tau},
	the measure $\mu$ is supported on $[0, 1/\sigma]$,
	and recall that we have assumed $\sigma = 1$.
	So consider
	\begin{align}
		\varepsilon_{k}^{-2(r + \nu)}
		\norm{ (T^* T)^r (x^\dag - x_k) }{}^2
	& =
		\varepsilon_k^{-2(r + \nu)}
		\int_{[0, 1]} \lambda^{2 r} (1 - \lambda)^{2 k } d\mu(\lambda)
		.
	\intertext{%
		In the case $k = 0$,
		this is majorized by $\mu([0, 1]) \leq \normiii{\mu}{\nu}^2$.
		Otherwise,
		defining $I(\lambda) := \mu([0, \lambda))$ 
		and integrating by parts as in 
		the proof of Lemma \ref{l:mu},
		we estimate
	}
	& =
		\varepsilon_k^{-2(r + \nu)}
		\int_{[0, 1]} (-(\lambda^{2 r} (1 - \lambda)^{2 k})' I(\lambda) d\lambda
	\\
	& \leq
		\varepsilon_k^{-2(r + \nu)}
		\int_{[0, 1]} 2 k \lambda^{2 r} (1 - \lambda)^{2 k - 1} I(\lambda) d\lambda
		.
	\end{align}
	Estimating $I(\lambda) \leq \lambda^{2 \nu} \normiii{\mu}{\nu}^2$,
	and employing Lemma \ref{l:c2}
	with $s := 2k$ and $a := 2 (r + \nu)$,
	the second inequality in \eqref{e:p:LW} is obtained.
	This completes the proof of \eqref{e:p:LW}.
\end{proof}

Unfortunately, the gap between our ``constants'' in \eqref{e:p:LW}
is not robust in $\nu$.
For example, if $r = 0$,
we have
$c_2 / c_1 \sim e^{\nu}$ as $\nu \to \infty$.
%
%
Qualitatively though,
\eqref{e:p:LW} means that
it is necessary and sufficient 
for
the asymptotic convergence rate $\nu$
that the data be in $\IX_\nu$,
which explains the notoriously slow convergence
of the Landweber iteration.
In the limit $\nu \to \infty$,
the upper estimate in \eqref{e:p:LW}
yields the asymptotics
\begin{align}
	\norm{ (T^* T)^r (x^\dag - x_k) }{}
	\leq
	c_2
	e^{-k}
	(1 + \tfrac12 k^2 \nu^{-1} + \mathcal{O}(\nu^{-2}))
	\normiii{ x^\dag - x_0 }{\nu}
	\quad\text{as}\quad
	\nu \to \infty
	,
\end{align}
for any fixed iteration $k \geq 0$.
This bound indicates that, for a fixed iteration $k$,
only a limited amount of smoothness can be exploited
(cf.~\cite[Section 3, Example 3]{Mathe2004}).

From \cite[Theorem 1 and Corollary 1]{FlemmingHofmannMathe2011},
an estimate similar to \eqref{e:p:LW} with $r = 0$
can be obtained
if we assume
$d(r) \sim r^{-\nu/(1 - \nu)}$ for sufficiently large $r > 0$
for the distance function
(see Section \ref{s:i:dr}).

We believe it is natural for $\normiii{\cdot}{\nu}$
to appear in \eqref{e:p:LW}
rather than the interpolation norm $\norm{\cdot}{\nu : \gamma}$.
The following example illustrates this in 
the limiting case $\nu = \gamma$.
As in the $k$-fold Tikhonov regularization,
however,
it might be possible and meaningful
to relate the error of the $k$-th iterate 
to the $\norm{\cdot}{\nu : k}$ norm of the data.

\begin{example}
\label{x:IX:LW}
	Let $T$ and $x^\dag$ be as in Example \ref{x:IX}.
	Recall that $\norm{x^\dag}{1} = \infty$,
	while $\normiii{x^\dag}{1} \approx 1.096$.
	We will estimate
	$\Delta_{\nu}^{0}(x^\dag)$ for $\nu := 1$.
	Since $\norm{T}{} = 1$, we set $\sigma := 1$ to satisfy \eqref{e:tau}.
	Then, for any integer $k \geq 0$,
	\begin{align}
	\label{e:xLW}
		\varepsilon_k^{-2}
		\norm{ x^\dag - x_k }{}^2
		=
		\varepsilon_k^{-2}
		\int_{[0, \infty)}
		(1 - \lambda)^{2 k}
		d \mu_{x^\dag}
		=
		(k + 1)^2
		\sum_{n \geq 1}
		n^{-3}
		(1 - n^{-1})^{2 k}
		.
	\end{align}
	Splitting the sum at $n = k$,
	and exploiting 
	in the finite subsum
	the fact that $t \mapsto t^{-3} (1 - t^{-1})^{2k}$
	peaks at $t = 1 + \tfrac23 k$
	we have
	\begin{align}
		\sum_{n \geq 1} (\cdots)
		\leq
		\sum_{n \leq k} (1 + \tfrac23 k)^{-3}
		+
		\sum_{n > k} n^{-3}
		\leq
		C
		(k + 1)^{-2}
		\quad
		\forall k \geq 1
	\end{align}
	for some $C \geq 0$ independent of $k$,
	meaning that \eqref{e:xLW}
	is bounded uniformly in $k \geq 0$.
	We find numerically that 
	$|\Delta_{\nu}^{0}(x^\dag)|^2 = \eqref{e:xLW}|_{k = 0} \approx (1.096)^2$,
	$\eqref{e:xLW}|_{k = 1} \approx (0.5453)^2$,
	$\eqref{e:xLW}|_{k = 2} \approx (0.5475)^2$,
	and
	$\eqref{e:xLW}$
	is decreasing for $k \geq 2$
	with 
	$\lim_{k \to \infty} \eqref{e:xLW} = (1/2)^2$.
	This limit can also be verified using \eqref{e:soo}.
	The infimum in $c_1$ is achieved at $k = 2$,
	so $c_1 = 1/3$.
	Thus \eqref{e:p:LW} reads 
	$1/3 \times 1.096 \leq 1.096 \leq 1.135 \times 1.096$.
	%
	%
	%
	%
	%
	If we use any of the iterates $(n^{-3/2} (1 - n^{-1})^{k})_{n \geq 1}$ as 
	the initial guess $x_0$,
	then \eqref{e:p:LW} reads
	$1/3 \times 1/\sqrt{2} \leq 1/2 \leq 1.135 \times 1/\sqrt{2}$.
\end{example}

%

%

%
%
Our primary motivation for introducing $r \geq 0$ in \eqref{e:LWE}
was to estimate the number of iterations needed
for
the stopping rule based on the discrepancy principle \eqref{e:MDPk}.
We briefly comment on this.
Fix $\tau > 1$.
In view of \eqref{e:dade} and \eqref{e:tau}
we have
\begin{align}
	\left|
		\norm{y^\dag - T x_k}{}
		-
		\norm{y^\delta - T x_k^\delta }{}
	\right|
	\leq
	\norm{ (y^\dag - T x_k) - (y^\delta - T x_k^\delta) }{}
	\leq
	\delta
\end{align}
for each iteration $k \geq 0$.
Hence, if 
$\norm{ y^\dag - T x_k }{} \leq \delta (\tau - 1)$
then \eqref{e:MDPk} is satisfied.
Using \eqref{e:p:LW} with $r = 1/2$,
this is fulfilled
for any $k \geq 0$ such that
$ 
	\delta (\tau - 1)
	\leq
	{\varepsilon}_k^{\nu + 1/2}
	c_2
	\normiii{x^\dag - x_0}{\nu}
$. 
%
%
This implies at most $k^\star \sim \delta^{-2 / (2 \nu + 1)}$ 
iterations
until \eqref{e:MDPk} is met,
as in \cite[Theorem 6.5]{EnglHankeNeubauer1996}
but with the slightly weaker assumption that 
the data 
is in $\IX_\nu$ rather than in $X_\nu$.
To reproduce the conclusion of \cite[Theorem 6.5]{EnglHankeNeubauer1996},
we still need to verify that
the discrepancy principle \eqref{e:MDPk} implies
the error rate
$\norm{ x^\dag - x_k^\delta }{} \lesssim \delta^{2 \nu / (2 \nu + 1)}$.
Here we cannot use 
the inequality 
\cite[(4.66)]{EnglHankeNeubauer1996},
\begin{align}
	\norm{ x^\dag - x_k }{}
	\leq
	\norm{ x^\dag - x_0 }{\nu}^{1 / (2 \nu + 1)} 
	\norm{ y^\dag - T x_k }{}^{2 \nu / (2 \nu + 1)}
	,
\end{align}
as in \cite[Proof of Theorem 6.5]{EnglHankeNeubauer1996}
for
the corresponding error rate
with data in $X_\nu$.
Instead, we estimate
\begin{align}
	\norm{ x^\dag - x_k }{}^2
	& =
	\int_{[0, \infty)} (1 - \sigma \lambda)^{2k} d\mu(\lambda)
\\
	& 
	\leq
	\int_{[0, \varepsilon_{k^\star})} d\mu(\lambda)
	+
	\varepsilon_{k^\star}^{-1}
	\int_{[\varepsilon_{k^\star}, \infty)} \lambda (1 - \sigma \lambda)^{2k} d\mu(\lambda)
\\
	& 
	\leq
	\varepsilon_{k^\star}^{2 \nu}
	\normiii{\mu}{\nu}^2
	+
	\varepsilon_{k^\star}^{-1}
	\norm{ y^\dag - T x_k }{}^2
	\lesssim
	\delta^{4 \nu / (2 \nu + 1)}
	,
\end{align}
and using this in the error splitting \eqref{e:LW:es} 
yields
$\norm{ x^\dag - x_k^\delta }{} \lesssim \delta^{2 \nu / (2 \nu + 1)}$.
%
%

%


\section{Conclusions}
\label{s:99}

We have 
introduced and investigated
families of Banach spaces 
and
their interrelations:
the Hilbert scale $X_\nu$, the interpolation spaces $X_{\nu:\gamma}$,
and
the spaces $\IX_\nu = X_{\nu : 2 \nu}$.
We have shown that 
the interpolation spaces $X_{\nu:\gamma}$
are most adequate for the characterization of convergence rates 
in (iterated) Tikhonov regularization,
while
$\IX_\nu$ are better suited
for the Landweber iteration.
\REV{%
This insight should facilitate
the verification 
of 
the convergence rates.
For instance,
in many situations,
$X_{\nu:\gamma}$
can be understood as (a subspace of)
the
Besov space $B_{2,\infty}^s$, 
for which many characterizations are known.
}

\section*{Acknowledgments}

I thank 
J.~Flemming for providing me with a copy of \cite{Flemming2012};
M.~Hansen for the reference to the Gagliardo completion;
O.~Scherzer for general comments on the initial draft;
the anonymous referees for their comments,
in particular
concerning quasioptimality \cite{RausHamarik2007} and Besov spaces.
The note was mainly written while at RICAM, Austrian Academy of Sciences, Linz (AT).
Supported in part by ANR-12-MONU-0013.

\section*{References}

\bibliographystyle{iopart-num}
\bibliography{refs}

\end{document}